\documentclass[11pt, reqno]{amsart}
\usepackage[utf8]{inputenc}
\usepackage{indentfirst, amssymb, amsmath, amsthm, mathrsfs, setspace, indentfirst, enumerate}
\usepackage[colorlinks=true,linkcolor=purple, citecolor=blue,urlcolor=magenta]{hyperref}
\usepackage{xcolor}
\usepackage{graphicx}
\usepackage{float}  
\newtheorem{theo}{Theorem}[section]
\newtheorem{lem}{Lemma}[section]

\newtheorem{defi}{Definition}[section]

\newcommand{\ol}{\overline}

\numberwithin{equation}{section}

\begin{document}

\title[Coefficient Problems and Sharp Determinant Estimates for $\mathcal{P}^{*}$]
{Coefficient Problems and Sharp Determinant Estimates for the Class $\mathcal{P}^{*}$}

\author[P. Das and N. Sarkar]{Pradip Das and Nabadwip Sarkar}
\address{Department of Mathematics, Raiganj University, Raiganj, West Bengal-733134, India.}
\email{pradipsmath@gmail.com}
\address{Amity School of Applied Sciences, Amity University Mumbai, Panvel, Navi Mumbai, Maharashtra-410206, India}
\email{nsarkar@mum.amity.edu, nabadwipsarkar52@gmail.com}
\renewcommand{\thefootnote}{}
\footnote{\emph{Corresponding Author}: Nabadwip Sarkar.}
\renewcommand{\thefootnote}{\arabic{footnote}}
\setcounter{footnote}{0}

\begin{abstract}
We study coefficient problems for the class $\mathcal{P}^{*}$ of normalized analytic functions $f$ in the unit disk satisfying the subordination condition
\[
f'(z) \prec e^z(1+\sin z), \qquad z\in\mathbb{D}.
\]
For functions in this class, we determine sharp bounds for the logarithmic coefficients $\gamma_n$ for $n=1,2,3,4$ and for the inverse logarithmic coefficients $\Gamma_n$ for $n=1,2,3$. In particular, we prove
\[
|\gamma_n| \le \frac{1}{n+1} \quad (n=1,2,3,4), \qquad
|\Gamma_1|,|\Gamma_2| \le \frac12,\quad |\Gamma_3|\le \frac{35}{48},
\]
with equality cases explicitly identified. Moreover, we establish sharp estimates for the moduli of the differences $|\gamma_2|-|\gamma_1|$ and $|\Gamma_2|-|\Gamma_1|$, yielding
\[
-\frac12 \le |\gamma_2|-|\gamma_1| \le \frac13,
\qquad
-\frac{1}{\sqrt{10}} \le |\Gamma_2|-|\Gamma_1| \le \frac13.
\]
Turning to determinant problems, we obtain sharp upper bounds for the second-order Hankel determinants associated with both the logarithmic and inverse logarithmic coefficients:
\[
\bigl|H_{2,1}(F_f/2)\bigr| \le \frac19,\qquad
\bigl|H_{2,1}(F_{f^{-1}}/2)\bigr| \le \frac{11}{96}.
\]
Finally, we derive the sharp two-sided estimate
\[
-\frac{9}{35} \le T_{3,1}(f) \le 1
\]
for the third-order Hermitian--Toeplitz determinant. All bounds are sharp, and the extremal functions are given explicitly. Our results extend and refine several known coefficient estimates for related subclasses of starlike and convex functions.
\end{abstract}

\thanks{Typeset by \AmS -\LaTeX}
\maketitle

\section{{\bf Preliminaries}}

Let $\mathcal{A}$ denote the class of normalized analytic functions $f$ defined on the open unit disk $\mathbb{D} := \{z : |z| < 1, \, z \in \mathbb{C}\}$, having the Taylor expansion
\begin{equation}\label{eqn1}
f(z) = z + \sum_{n=2}^{\infty} a_n(f) z^n, \quad z \in \mathbb{D}.
\end{equation}

Let $\mathcal{S}$ denote the subclass of $\mathcal{A}$ consisting of univalent functions. For each $f \in \mathcal{S}$, let $g = f^{-1}$ denote its inverse, which is analytic in a neighborhood of the origin and admits the series expansion
\begin{equation}\label{eqn2}
g(w) = f^{-1}(w) = w + \sum_{n=2}^{\infty} A_n w^n.
\end{equation}

By Koebe's one-quarter theorem, the radius of convergence of this series is at least $|w| < 1/4$. Applying variational methods, L\"owner \cite{Loewner1923} established the sharp estimate
\[
|A_n| \leq K_n, \quad \text{for all } n \in \mathbb{N},
\]
where $K_n = \frac{(2n)!}{n!(n+1)!}$, and $K(w) = w + \sum_{n=2}^{\infty} K_n w^n$ denotes the inverse of the Koebe function.

For $f \in \mathcal{S}$, the identity $f(f^{-1}(w)) = w$, together with the expansion in (\ref{eqn2}), yields the following expressions for the initial inverse coefficients:
\begin{equation}\label{eqn3}
\begin{cases}
A_2 = -a_2, \\
A_3 = -a_3 + 2a_2^2, \\
A_4 = -a_4 + 5a_2a_3 - 5a_2^3.
\end{cases}
\end{equation}

The logarithmic coefficients $\gamma_n$ of a function $f \in \mathcal{S}$ are defined by the relation
\begin{equation}\label{eqn4}
F_f(z) := \log\frac{f(z)}{z} = 2 \sum_{n=1}^{\infty} \gamma_n z^n, \quad z \in \mathbb{D}.
\end{equation}

Determining sharp upper bounds for $\gamma_n$ is generally a difficult problem. Milin \cite{IMM1} emphasized the importance of these coefficients in connection with the Bieberbach conjecture and proposed the following inequality for $f \in \mathcal{S}$ and $n \geq 2$:
\[
\sum_{m=1}^n \sum_{k=1}^m \left(k|\gamma_k|^2 - \frac{1}{k}\right) \leq 0.
\]
This was later established by de Branges \cite{LDB1}, thereby resolving the Bieberbach conjecture. For the Koebe function $k(z) = z/(1-z)^2$, the logarithmic coefficients are $\gamma_n = 1/n$. Although the Koebe function often serves as an extremal example, the conjecture $|\gamma_n| \leq 1/n$ is not valid in general \cite{PLD1}.

Differentiating (\ref{eqn4}) and comparing coefficients yields the following expressions for $\gamma_n$ in terms of $a_n$:
\begin{equation}\label{log1}
\begin{cases}
\gamma_1 = \frac{1}{2} a_2, \\
\gamma_2 = \frac{1}{2}\left(a_3 - \frac{1}{2}a_2^2\right), \\
\gamma_3 = \frac{1}{2}\left(a_4 - a_2a_3 + \frac{1}{3}a_2^3\right), \\
\gamma_4 = \frac{1}{2}\left(a_5 - a_2 a_4 + a_2^2 a_3 - \frac{1}{2} a_3^2 - \frac{1}{4} a_2^4\right).
\end{cases}
\end{equation}

The estimate $|\gamma_1| \leq 1$ is elementary, while the sharp bound for $\gamma_2$ is given by $|\gamma_2| \leq \frac{1}{2}(1 + 2e^{-2}) \approx 0.635$. For $n \geq 3$, general sharp bounds remain unknown, although progress has been made for various subclasses \cite{AA1, AA2, RK1}.

The logarithmic inverse coefficients $\Gamma_n$, introduced by Ponnusamy et al. \cite{Ponnusamy2018}, are defined via the inverse function $f^{-1}$:
\begin{equation}
F_{f^{-1}}(w) := \log\frac{f^{-1}(w)}{w} = 2 \sum_{n=1}^{\infty} \Gamma_n w^n, \quad |w| < \frac{1}{4}.
\end{equation}
The first three coefficients are given by
\begin{equation}\label{IG1}
\begin{cases}
\Gamma_1 = -\frac{1}{2} a_2, \\
\Gamma_2 = -\frac{1}{2} a_3 + \frac{3}{4} a_2^2, \\
\Gamma_3 = -\frac{1}{2} \left( a_4 - 4a_2 a_3 + \frac{10}{3} a_2^3 \right).
\end{cases}
\end{equation}

Ponnusamy et al. \cite{Ponnusamy2018} proved that for the class $\mathcal{S}$,
\[
|\Gamma_n| \leq \frac{1}{2n} \binom{2n}{n}, \quad n \in \mathbb{N},
\]
with equality only for the Koebe function and its rotations. They also obtained sharp bounds for the initial logarithmic inverse coefficients for several important geometric subclasses of $\mathcal{S}$.

Kowalczyk and Lecko \cite{12} recently initiated the study of the Hankel determinant whose entries are the logarithmic coefficients of $f \in \mathcal{S}$, defined by
\begin{eqnarray}\label{Hank1}
H_{q,n}\left(F_f / 2\right) = \left|\begin{array}{cccc}
\gamma_n & \gamma_{n+1} & \cdots & \gamma_{n+q-1} \\
\gamma_{n+1} & \gamma_{n+2} & \cdots & \gamma_{n+q} \\
\vdots & \vdots & \ddots & \vdots \\
\gamma_{n+q-1} & \gamma_{n+q} & \cdots & \gamma_{n+2(q-1)}
\end{array}\right|.
\end{eqnarray}

Analogously, the Hankel determinant $H_{q,n}\left(F_{f^{-1}} / 2\right)$ \cite{A1}, whose entries are the logarithmic coefficients of the inverse function, is given by
\begin{eqnarray}\label{Hank2}
H_{q,n}\left(F_{f^{-1}} / 2\right) = \left|\begin{array}{cccc}
\Gamma_n & \Gamma_{n+1} & \cdots & \Gamma_{n+q-1} \\
\Gamma_{n+1} & \Gamma_{n+2} & \cdots & \Gamma_{n+q} \\
\vdots & \vdots & \ddots & \vdots \\
\Gamma_{n+q-1} & \Gamma_{n+q} & \cdots & \Gamma_{n+2(q-1)}
\end{array}\right|.
\end{eqnarray}

It is also of interest to study estimates for the Hermitian--Toeplitz determinant. For $q,n \in \mathbb{N}$, the Hermitian--Toeplitz determinant (see \cite{hh,hh1}) of order $q$ associated with the coefficient sequence $\{a_k\}_{k\geq 1}$ of a function $f \in \mathcal{A}$ is defined as
\begin{equation}\label{HTdef}
T_{q,n}(f) :=
\begin{vmatrix}
a_n & a_{n+1} & \cdots & a_{n+q-1} \\
a_{n+1} & a_n & \cdots & a_{n+q-2} \\
\vdots & \vdots & \ddots & \vdots \\
a_{n+q-1} & a_{n+q-2} & \cdots & a_n
\end{vmatrix}.
\end{equation}

From (\ref{HTdef}), a direct computation yields the following third-order Hermitian--Toeplitz determinant:
\begin{equation}\label{HTexamples}
T_{3,1}(f) = 2\Re\!\left(a_2^2 \overline{a_3}\right) - 2|a_2|^2 - |a_3|^2 + 1.
\end{equation}

The investigation of Hermitian--Toeplitz determinants for various subclasses of normalized analytic functions was initiated in \cite{C,K} and subsequently extended in \cite{CKS1}. More recent contributions in this direction appear in \cite{K3,K2}. In particular, Cudna et al. \cite{C} established sharp bounds for the second- and third-order Hermitian--Toeplitz determinants for starlike and convex functions of order $\beta$. Later, Kumar et al. \cite{K3} obtained sharp estimates for the same determinants in the context of Janowski starlike and convex functions, thereby generalizing the results of \cite{C}. Additional related findings can be found in \cite{CK,VK,VKN,VKS}.

In recent years, considerable attention has been devoted to the study of Hankel determinants involving logarithmic coefficients for various subclasses of analytic functions, including starlike, convex, univalent, strongly starlike, and strongly convex functions (see \cite{12,KL2,KL3,STZ,ref35,ref36,ref37} and the references therein).

Before presenting the main results of this paper, we recall the concept of differential subordination, a fundamental tool in geometric function theory that provides an effective framework for investigating subclasses of $\mathcal{A}$.

\begin{defi}
Let $\Omega$ denote the class of analytic functions $\omega$ in $\mathbb{D}$ satisfying $\omega(0)=0$ and $|\omega(z)|<1$ for all $z \in \mathbb{D}$. Functions in $\Omega$ are called Schwarz functions. Each $\omega \in \Omega$ admits a power series expansion
\[
\omega(z) = \sum_{n=1}^{\infty} \omega_n z^n, \qquad z \in \mathbb{D}.
\]

For two analytic functions $f$ and $g$ on $\mathbb{D}$, we say that $f$ is subordinate to $g$ in $\mathbb{D}$, written $f \prec g$, if there exists a Schwarz function $\omega \in \Omega$ such that $f(z) = g(\omega(z))$ for all $z \in \mathbb{D}$. In particular, if $g$ is univalent in $\mathbb{D}$, then $f \prec g$ if and only if $f(0)=g(0)$ and $f(\mathbb{D}) \subset g(\mathbb{D})$.
\end{defi}

Tang et al. \cite{TAH} introduced the following class:
\[
\mathcal{P}^{*}
=
\left\{
f \in \mathcal{S} : f'(z) \prec e^z(1+\sin z), \quad z \in \mathbb{D}
\right\}.
\]

Several coefficient problems for functions in the class $\mathcal{P}^{*}$ were investigated in \cite{TAA}. In the present work, we address various coefficient problems for analytic functions in $\mathcal{P}^{*}$. We derive sharp estimates for the logarithmic coefficients and their inverse counterparts, including bounds for their differences. In addition, we investigate the second-order Hankel and Hermitian--Toeplitz determinants associated with these functions.

Let $\mathcal{P}$ denote the class of analytic functions $p$ on $\mathbb{D}$ satisfying $p(0)=1$ and $\Re p(z) > 0$ for all $z \in \mathbb{D}$. Every $p \in \mathcal{P}$ admits the series representation
\begin{equation}\label{p1}
p(z) = 1 + \sum_{n=1}^{\infty} c_n z^n, \quad z \in \mathbb{D}.
\end{equation}

Functions in $\mathcal{P}$ are called Carath\'{e}odory functions. It is well known that for $p \in \mathcal{P}$, the coefficients satisfy the sharp bound $|c_n| \leq 2$ for all $n \geq 1$ (see \cite{PLD1}). The Carath\'{e}odory class $\mathcal{P}$ and its coefficient estimates play a central role in deriving sharp bounds in geometric function theory.

The paper is organized as follows:

\begin{itemize}
\item In Section~2, we state some lemmas, which will be useful to establish our main results.

\item In Section~3.1, we obtain sharp bounds for the logarithmic coefficients $\gamma_n$ ($n=1,2,3,4$) of functions belonging to the class $\mathcal{P}^{*}$.

\item Section~3.2 is devoted to deriving sharp bounds for the logarithmic inverse coefficients $\Gamma_n$ ($n=1,2,3$) of functions in the same class.

\item In Section~3.3, we establish estimates for the differences between the logarithmic coefficients and the logarithmic inverse coefficients, namely $|\gamma_2|-|\gamma_1|$ and $|\Gamma_2|-|\Gamma_1|$, for functions in $\mathcal{P}^{*}$.

\item Section~3.4 deals with the second-order Hankel determinant $H_{2,1}(F_f/2)$ associated with the logarithmic coefficients and the second-order Hankel determinant $H_{2,1}(F_{f^{-1}}/2)$ associated with the logarithmic inverse coefficients of functions in the class $\mathcal{P}^{*}$.

\item Finally, Section~3.5 focuses on the third-order Hermitian--Toeplitz determinant $T_{3,1}(f)$ associated with functions belonging to the class $\mathcal{P}^{*}$.
\end{itemize}

\section{{\bf Auxiliary lemmas}}

Now we recall the following well-known result due to Cho et al. \cite{C12}.

\begin{lem}\label{L1} \cite[Lemma 2.4]{C12} If $p\in\mathcal{P}$ is of the form (\ref{p1}), then
\begin{eqnarray}
\label{c1}c_1 =2\tau_1,
\end{eqnarray}
\begin{eqnarray}\label{c2} c_2=2\tau_1^2 + 2(1 - |\tau_1|^2)\tau_2,
\end{eqnarray}
\begin{eqnarray}
\label{c3} c_3 = 2\tau_1^3+4(1-|\tau_1|^2)\tau_1\tau_2 - 2(1 - |\tau_1|^2)\ol{\tau_1}\tau_2^2 + 2(1 - |\tau_1|^2)(1 - |\tau_2|^2)\tau_3,
\end{eqnarray}
and
\begin{align}
\label{c4}
c_4={}&2\tau_1^4+4(1-|\tau_1|^2)\tau_1^2\tau_2
+2(1-|\tau_1|^2)\tau_2^2 \nonumber\\
&+4(1-|\tau_1|^2)(1-|\tau_2|^2)\tau_1\tau_3 \nonumber\\
&+2(1-|\tau_1|^2)(1-|\tau_2|^2)(1-|\tau_3|^2)\tau_4.
\end{align}
for some $\tau_1\in [0,1]$ and $\tau_2, \tau_3, \tau_4 \in \mathbb{\ol D}:= \{ z \in \mathbb{C} : |z| \leq 1 \}$.

For $\tau_1 \in \mathbb{T} := \{ z \in \mathbb{C} : |z| = 1 \}$, there is a unique function $p \in \mathcal{P}$ with $c_1$ as in (\ref{c1}), namely,
\[
p(z) = \frac{1 + \tau_1 z}{1 - \tau_1 z}, \quad z \in \mathbb{D}.
\]

For $\tau_1 \in \mathbb{D}$ and $\tau_2 \in \mathbb{T}$, there is a unique function $p \in \mathcal{P}$ with $c_1$ and $c_2$ as in (\ref{c1}) and (\ref{c2}), namely,
\[
p(z) = \frac{1 + (\ol \tau_1 \tau_2 + \tau_1) z + \tau_2 z^2}{1 + (\ol \tau_1 \tau_2 - \tau_1) z - \tau_2 z^2}, \quad z \in \mathbb{D}.
\]

For $\tau_1, \tau_2 \in \mathbb{D}$ and $\tau_3 \in \mathbb{T}$, there is a unique function $p \in \mathcal{P}$ with $c_1$, $c_2$, and $c_3$ as in (\ref{c1}-\ref{c3}), namely,
\[
p(z) = \frac{1 + (\ol\tau_2 \tau_3 + \ol\tau_1 \tau_2 + \tau_1)z+(\ol\tau_1\tau_3+\tau_1\ol\tau_2\tau_3+\tau_2)z^2+\tau_3z^3}{1+(\ol\tau_2\tau_3+\ol\tau_1\tau_2-\tau_1)z+(\ol\tau_1\tau_3-\tau_1\ol\tau_2\tau_3-\tau_2)z^2-\tau_3z^3},\;\;z\in\mathbb{D}.
\]
\end{lem}

\medskip

The following well-known result is due to Choi et al. \cite{CKS1}.

\begin{lem}\label{L2}\cite{CKS1} Let $A$, $B$, $C$ be real numbers and let
\[Y(A, B, C):= \max\limits_{z\in \ol{\mathbb{D}}}\left\lbrace |A+Bz+Cz^2|+1-|z|^2\right\rbrace.\]

\begin{enumerate} 
\item[(i)] If $AC\geq 0$, then
\[Y(A, B, C) =
\begin{cases}
|A|+|B|+|C|, & \text{if}\;\;\; |B|\geq 2(1-|C|), \\
1+|A|+\frac{B^2}{4(1-|C|)}, &\text{if}\;\;\; |B|<2(1-|C|).
\end{cases}
\]
\item[(ii)] If $AC<0$, then 
\[Y(A,B,C)=
\begin{cases}
1-|A|+\frac{B^2}{4(1-|C|)}, &\text{if}\;\;\; -4AC(C^{-2}-1) \leq B^2\; \text{and}\; |B|<2(1-|C|), \\
1+|A|+\frac{B^2}{4(1+|C|)}, &\text{if}\;\;\; B^2<\min\left\{4(1+|C|)^2, -4AC(C^{-2}-1) \right\}, \\
R(A,B,C), &\text{otherwise},
\end{cases}
\]
where
\[R(A,B,C):=
\begin{cases}
|A|+|B|-|C|, & \text{if}\;\;\; |C|(|B|+4|A|) \leq |AB|, \\
-|A|+|B|+|C|, & \text{if}\;\;\; |AB|\leq |C|(|B|-4|A|), \\
(|C|+|A| )\sqrt{1-\frac{B^2}{4AC}}, &\text{otherwise}.
\end{cases}
\]
\end{enumerate} 
\end{lem}

\begin{lem}\label{L3} \cite{MM}
Let $p \in \mathcal{P}$ be given by \eqref{p1}. Then
\[
\left| c_2 - v c_1^2 \right| \le 
\begin{cases}
-4v + 2, & v < 0, \\
2, & 0 \leq v \leq 1, \\
4v - 2, & v > 1.
\end{cases}
\]

Moreover, for $v < 0$ or $v > 1$, equality holds if and only if
\[
p(z) = \frac{1+z}{1-z} \quad \text{or one of its rotations}.
\]

For $0 < v < 1$, equality holds if and only if
\[
p(z) = \frac{1+z^2}{1-z^2} \quad \text{or one of its rotations}.
\]
\end{lem}

\begin{lem}\label{L4}\cite{A1}
Let $p \in \mathcal{P}$ be given by \eqref{p1} with $0 \leq B \leq 1$ and $B(2B - 1) \leq D \leq B$. Then
\[
\left| c_3 - 2 B c_1 c_2 + D c_1^3 \right| \leq 2.
\]
\end{lem}

\begin{lem}\label{L5}\cite{Ravichandran2015}
Let $p \in \mathcal{P}$ be given by \eqref{p1}. If $\alpha, \beta, \gamma, \lambda$ satisfy
\[
0 < \alpha < 1, \quad 0 < \lambda < 1,
\]
and
\begin{align*}
	8&\lambda(1 - \lambda) \Big\{ (\alpha\beta - 2\gamma)^2 + (\alpha(\lambda + \alpha) - \beta)^2 \Big\}
	+ \alpha(1 - \alpha)(\beta - 2\lambda \alpha)^2\\
	& \leq 4 \alpha^2 (1 - \alpha)^2 \lambda (1 - \lambda),
\end{align*}
then
\[
|\gamma c_1^4 + \lambda c_2^2 + 2 \alpha c_1 c_3 - \frac{3}{2} \beta c_1^2 c_2 - c_4 |\leq 2.
\]
\end{lem}

\begin{lem}\label{L6}\cite{SimThomas2020}
Let $J, K,$ and $L$ be numbers such that $J \geq 0$, $K \in \mathbb{C}$, and $L \in \mathbb{R}$. 
Let $p \in \mathcal{P}$ be of the form (\ref{p1}) and define a function by

\[
\Phi(c_1,c_2) = \big| K c_1^2 + L c_2 \big| - \big| J c_1 \big|.
\]
Then 
\[
\Phi(c_{1}, c_{2}) \le 
\begin{cases}
|4K + 2L| - 2J, & \text{if } |2K + L| \geq |L| + J, \\[6pt]
2|L|, & \text{otherwise.}
\end{cases}
\]
and
\[
-\Phi(c_1,c_2) \leq 
\begin{cases}
2J - M, & \text{when } J \geq M + 2|L|, \\[6pt]
2J \sqrt{\dfrac{2|L|}{M + 2|L|}}, & \text{when } J^2 \leq 2|L|(M + 2|L|), \\[10pt]
2|L| +\dfrac{ J^2}{M + 2|L|}, & \text{otherwise}
\end{cases}
\]
where $M=|4K+2L|$.
\end{lem}

\section{{\bf Main theorems and their proofs}}

\subsection{{\bf Sharp bounds for the logarithmic coefficients of the class $\mathcal{P}^{*}$}} 

\begin{theo}\label{T1} 
Let $f \in \mathcal{P}^{*}$ and $\gamma_n \ (n = 1, 2, 3,4)$ be defined by (\ref{log1}). Then 
\[
|\gamma_n| \leq \frac{1}{n+1}, \qquad n = 1, 2, 3,4.
\]
The inequalities are sharp.
\end{theo}

\begin{proof} 
Since $f \in \mathcal{P}^{*}$ then there exists a Schwarz function $w(z)$ with 
$w(0) = 0$ and $|w(z)| < 1$ in $\mathbb{D}$ such that  
\begin{equation}\label{t1.1}
f'(z) = e^{w(z)}\big(1 + \sin w(z)\big).
\end{equation}

If $p \in \mathcal{P}$, then we can write
\begin{equation}\label{t1.2}
w(z) = \frac{p(z)-1}{p(z)+1}.
\end{equation}

Let $p$ be given by (\ref{p1}). From \eqref{t1.1} and \eqref{t1.2}, by equating the coefficients we obtain
\begin{eqnarray}\label{a2}
\left.
\begin{aligned}
a_2 &= \frac{1}{2} c_1,\\[6pt]
\label{a3}a_3 &= \frac{1}{3} c_2 - \frac{c_1^2}{24},\\[6pt]
\label{a4}a_4 &= -\frac{c_1c_2}{16}-\frac{c_1^3}{64}+\frac{c_3}{4},\\[6pt]
\label{a5}a_5 &= -\frac{3c_1^2c_2}{80} -\frac{c_1c_3}{20}
               -\frac{c_2^2}{40}+\frac{c_4}{5}
               +\frac{5c_1^4}{384}
\end{aligned}
\right\}
\end{eqnarray}

\begin{enumerate}
  \item[(i)] {\bf Sharp bound of $\gamma_1$:} By using (\ref{log1}) and (\ref{a2}) we have 
 \[|\gamma_1|=\left|\frac{1}{2}a_2\right|=\frac{1}{4}|c_1|\leq \frac{1}{2}.\]
 The inequality is sharp for the function $f_1$, which is defined by 
\begin{eqnarray}\label{f1} 
f_1(z)
= \int_{0}^{z} (1+\sin t)\, e^{t}\, dt
= z + z^{2} + \frac{1}{2} z^{3}
+ \frac{1}{8} z^{4}
+ \frac{1}{120} z^{5}
+ \cdots .
\end{eqnarray}
For this function, $c_1=2$, $c_2=2$, $c_3=2$, and $|\gamma_1|=1/2$.

\item[(ii)] {\bf Sharp bound of $\gamma_2$:} From (\ref{log1}), (\ref{a2}) and (\ref{a3}), we get 
\begin{eqnarray*} |\gamma_2|&=&\left| \frac{1}{2}\left(a_3 - \frac{1}{2}a_2^2\right)\right|\\
&=& \frac{1}{2}\left|\left(\frac{c_2}{3} -\frac{c_1^2}{24}\right)-\frac{1}{2}\frac{c_1^2}{4}\right|\\
&=& \frac{1}{6}\left|c_2-\frac{1}{2}c_1^2\right|\\
&=&\frac{1}{6}|c_2-vc_1^2|,\end{eqnarray*}
where $v=\frac{1}{2}$. Therefore applying Lemma \ref{L3}, we deduce that 
\[|\gamma_2|\leq \frac{1}{6}\cdot 2=\frac{1}{3}.\]

The inequality is sharp for the function $f_2$, which is defined by 
\begin{eqnarray}\label{f2} 
f_2(z)
= \int_{0}^{z} (1+\sin t^2)\, e^{t^2}\, dt
= z + \frac{2}{3}z^{3} + \frac{3}{10}z^{5}
+ \frac{1}{14} z^{7}
+ \frac{1}{216} z^{9}
+ \cdots .
\end{eqnarray}
For this function, $c_1=0$, $c_2=2$, and $|\gamma_2|=1/3$.

\item[(iii)] {\bf Sharp bound of $\gamma_3$:} From (\ref{log1}), (\ref{a2})-(\ref{a4}), we have
\begin{eqnarray*} |\gamma_3|&=&\left| \frac{1}{2}\left(a_4 - a_2a_3 + \frac{1}{3}a_2^3\right)\right|\\
                &=& \frac{1}{8} \left|c_3-\frac{11}{12}c_1c_2+\frac{3}{16}c_1^3\right|\\
&=& \frac{1}{8}|c_3-2Bc_1c_2+Dc_1^3|,\end{eqnarray*}
where $B=\frac{11}{24}$ and $D=\frac{3}{16}$. Therefore it is clear that $0\leq B\leq 1$ and the inequality $B(2B-1)\leq D\leq B$ holds. Indeed,
\[
B(2B-1)=\frac{11}{24}\left(\frac{22}{24}-1\right)=\frac{11}{24}\left(-\frac{2}{24}\right)=-\frac{11}{288}\leq \frac{3}{16},
\]
and $\frac{3}{16}\leq \frac{11}{24}$.

Now applying Lemma \ref{L4} we deduce that 
\[ |\gamma_3|\leq \frac{1}{8}\cdot 2=\frac{1}{4}.\]

The inequality is sharp for the function $f_3$, which is defined by 
\begin{eqnarray}\label{f3} 
f_3(z)
= \int_{0}^{z} (1+\sin t^3)\, e^{t^3}\, dt
= z + \frac{1}{2}z^{4} + \frac{3}{14}z^{7}
+ \cdots .
\end{eqnarray}
For this function, $c_1=0$, $c_2=0$, $c_3=2$, and $|\gamma_3|=1/4$.

\item[(iv)] {\bf Sharp bound of $\gamma_4$:} From (\ref{log1}), (\ref{a2})-(\ref{a5}), we have
\begin{eqnarray}\label{G4} 
\begin{aligned}
|\gamma_4|
&=\left| \frac{1}{2}\left(a_5 - a_2 a_4 + a_2^2 a_3 - \frac{1}{2} a_3^2 - \frac{1}{4} a_2^4\right)\right| \\[6pt]
&=\left|-\frac{7}{2304}c_1^4+\frac{131}{2880}c_1^2c_2-\frac{7}{80}c_1c_3-\frac{29}{720}c_2^2+\frac{1}{10}c_4\right| \\[6pt]
&=\frac{1}{10}\left| \frac{35}{1152}c_1^4+\frac{29}{72}c_2^2+\frac{7}{8}c_1c_3-\frac{131}{288}c_1^2c_2-c_4\right| \\[6pt]
&=\frac{1}{10}\left| \gamma\, c_1^4+\lambda\, c_2^2+2\alpha\, c_1 c_3-\frac{3}{2}\beta\, c_1^2 c_2-c_4 \right|,
\end{aligned} \end{eqnarray}
where
\[
\alpha=\frac{7}{16},\qquad \lambda=\frac{29}{72},\qquad \beta=\frac{131}{432},\qquad \gamma=\frac{35}{1152}.
\]

A direct calculation gives
\begin{align*}
&8\lambda(1 - \lambda) \Big\{ (\alpha\beta - 2\gamma)^2 + (\alpha(\lambda + \alpha) - \beta)^2 \Big\} + \alpha(1 - \alpha)(\beta - 2\lambda \alpha)^2 \\
&= \frac{1}{2^{16} \cdot 3^8 \cdot 5 \cdot 7} \cdot 1693945 \leq \frac{1}{2^{12} \cdot 3^4 \cdot 5 \cdot 7} = 4\alpha^2(1-\alpha)^2\lambda (1-\lambda),
\end{align*}
which can be verified by direct simplification.

Thus the condition of Lemma \ref{L5} holds. Therefore applying Lemma \ref{L5} to (\ref{G4}) we deduce that
\[ |\gamma_4|\leq \frac{1}{10}\cdot 2=\frac{1}{5}.\]

The inequality is sharp for the function $f_4$, which is defined by 
\begin{eqnarray}\label{f4} 
f_4(z)
= \int_{0}^{z} (1+\sin t^4)\, e^{t^4}\, dt
= z + \frac{2}{5}z^{5} + \frac{1}{6}z^{9}
+ \frac{1}{26} z^{13}
+ \frac{1}{408} z^{17}
+ \cdots .
\end{eqnarray}
For this function, $c_1=c_2=c_3=0$, $c_4=2$, and $|\gamma_4|=1/5$.
 \end{enumerate}
 
\end{proof}

\subsection{{\bf Sharp bounds for the logarithmic inverse coefficients of the class $\mathcal{P}^{*}$}} 

\begin{theo}\label{T2} Let \( f \in \mathcal{P}^{*} \), and let the coefficients \( \Gamma_n \) \((n=1,2,3)\) be defined by (\ref{IG1}). Then
\[
|\Gamma_n| \le \frac{1}{2}, \qquad n=1,2,
\]
and
\[
|\Gamma_3| \le \frac{35}{48}.
\]
The inequalities are sharp.
\end{theo}

\begin{proof} Let \( f \in \mathcal{P}^{*} \).

\begin{enumerate}

\item[(i)] \textbf{Sharp bound of $\Gamma_1$.}
Using \eqref{IG1} and \eqref{a2}, we obtain
\[
|\Gamma_1|
= \left|-\frac{1}{2}a_2\right|
= \frac{1}{4}|c_1|
\le \frac{1}{2}.
\]
The inequality is sharp for the function $f_1$ defined in~\eqref{f1}.

\item[(ii)] {\bf Sharp bound of $\Gamma_2$:} From \eqref{IG1} and \eqref{a2}, we obtain
\begin{eqnarray*}
|\Gamma_2|
&=&\left|-\frac{1}{2}a_3+\frac{3}{4}a_2^2\right| \\
&=&\left|-\frac{1}{6}c_2+\frac{1}{48}c_1^2+\frac{9}{48}c_1^2\right| \\
&=&\frac{1}{6}\left|c_2-\frac{5}{4}c_1^2\right| \\
&=&\frac{1}{6}\left|c_2-vc_1^2\right|,
\end{eqnarray*}
where \(v=\frac{5}{4}>1\). Therefore, by applying Lemma~\ref{L3}, we obtain
\[
|\Gamma_2|\leq \frac{1}{6}(4v-2)=\frac{1}{6}(5-2)=\frac{1}{2}.
\]
The inequality is sharp, and equality is attained by the function \(f_1\), defined in \eqref{f1}. Indeed, for $f_1$, $c_1=2$ and $c_2=2$, so $\left|c_2-\frac{5}{4}c_1^2\right|=|2-5|=3$, yielding $|\Gamma_2|=1/2$.

\item[(iii)] {\bf Sharp bound of $\Gamma_3$:} From \eqref{IG1} and \eqref{a2}, we have
\begin{eqnarray}
\label{G3}
|\Gamma_3|
&=&
\left|-\frac{1}{2}\left(a_4-4a_2a_3+\frac{10}{3}a_2^3\right)\right| \nonumber\\
&=&
\left|
\frac{c_1c_2}{32}
+\frac{c_1^3}{128}
-\frac{c_3}{8}
+\frac{c_1c_2}{3}
-\frac{c_1^3}{24}
-\frac{5}{24}c_1^3
\right| \nonumber\\
&=&
\frac{1}{8}
\left|
c_3+\frac{31}{16}c_1^3-\frac{35}{12}c_1c_2
\right|.
\end{eqnarray}

Now, substituting the expressions for \(c_1\), \(c_2\), and \(c_3\) from \eqref{c1}--\eqref{c3} into \eqref{G3}, we obtain
\begin{eqnarray}
\label{G11}
|\Gamma_3|
&=&
\frac{1}{8}
\left|
\frac{35}{6}\tau_1^3
+
(1-\tau_1^2)
\left(
-\frac{23}{3}\tau_1\tau_2
-2\tau_1\tau_2^2
+2\tau_3(1-|\tau_2|^2)
\right)
\right|.
\end{eqnarray}

Note that for
\[
f_\theta(z):=e^{-i\theta}f(e^{i\theta}z), \qquad \theta\in\mathbb{R},
\]
where \(f\in\mathcal{S}\), we have
\[
\Gamma_3(f_\theta)=e^{3i\theta}\Gamma_3(f).
\]
Thus, the quantity \(|\Gamma_3|\) is invariant under rotations of \(f\). Moreover, since \(p\in\mathcal{P}\) is a Carath\'{e}odory function, we have
\[
|c_n|\leq 2,\qquad n=1,2,\ldots.
\]
Since the class \(\mathcal{P}\) is invariant under rotations, we may assume that
\[
c_1\in[0,2].
\]
Consequently,
\[
\tau_1\in[0,1].
\]

We now consider the following three cases according to the value of \(\tau_1\).

{\bf Case 1.} Let $\tau_1=1$. Then from (\ref{G11}) we easily obtain
\[|\Gamma_3|=\frac{35}{48}.\]

{\bf Case 2.} Let $0\leq \tau_1<1$. Applying the triangle inequality in (\ref{G11}) and using the fact that $|\tau_1|\leq 1$, we obtain
\begin{eqnarray}\label{G12} |\Gamma_3| &\leq & \frac{(1 - \tau_1^2)}{4}\left(\left|\frac{35}{12}\frac{\tau_1^3}{(1 - \tau_1^2)}-\frac{23}{6}\tau_1\tau_2 - \tau_1\tau_2^2\right| +(1 -|\tau_2|^2)\right)\nonumber\\
&\leq&\frac{(1 - \tau_1^2)}{4} (|A+B\tau_2+C\tau_2^2|+1-|\tau_2|^2),\end{eqnarray}
where $A=\frac{35\tau_1^3}{12(1 - \tau_1^2)}$, $B=-\frac{23}{6}\tau_1$ and $C=-\tau_1$.

Observe that $AC<0$. Hence we can apply case (ii) of Lemma \ref{L2}. Next, we check all the conditions of case (ii).

\begin{enumerate}
\item[(a)] A simple computation shows that the inequality
\[ \frac{529}{36} \tau_1^2=B^2 < 4(1+|C|)^2 \] is true for all $\tau_1\in [0,1).$
Note that $\frac{529}{36}\tau_1^2=B^2 < -4AC(C^{-2}-1)=\frac{35}{3}\tau_1^2$, is false for $\tau_1\in [0,1)$ since $529/36 > 35/3$. Therefore the condition for the second sub-case of (ii) is: $$B^2 < \min\left\{ 4(1 + |C|)^2, \frac{35}{3}\tau_1^2 \right\}.$$ Because $B^2$ is actually greater than $\frac{35}{3}\tau_1^2$, the inequality fails. Therefore, this specific sub-case does not hold for our problem.

\item[(b)]
Here,
\[
|B|-2(1-|C|)=35\tau_1-12<0
\]
holds for
\[
\tau_1<\frac{12}{35}\approx 0.342857.
\]
Moreover,
\[
-4AC(C^{-2}-1)-B^2
=-\frac{109}{36}\tau_1^2\leq 0
\]
for \(0\leq\tau_1<1\).

Applying Lemma~\ref{L2} to \eqref{G12}, we obtain
\begin{eqnarray*}
|\Gamma_3|
&\leq&
\frac{1-\tau_1^2}{4}
\left(
1-|A|+\frac{B^2}{4(1-|C|)}
\right) \\
&=&
\frac{1-\tau_1^2}{4}\,Y(\tau_1),
\end{eqnarray*}
where
\[
Y(\tau_1)
=
1+\frac{385\tau_1^2+109\tau_1^3}{144(1-\tau_1^2)}.
\]

Furthermore,
\[
Y'(\tau_1)
=
\frac{\tau_1\left(770+327\tau_1-1097\tau_1^3\right)}
{144(1-\tau_1^2)^2}.
\]
Since the numerator is positive on the interval
\(\left(0,\frac{12}{35}\right)\),
it follows that \(Y\) is increasing there. Hence,
\[
\max_{\,0\leq\tau_1\leq 12/35}Y(\tau_1)
=
Y\!\left(\frac{12}{35}\right)
=
1+\frac{385(144/1225)+109(1728/42875)}{144(1081/1225)}
=
\frac{10031}{7218}
\approx
1.3897.
\]
Consequently,
\[
|\Gamma_3|
\leq
0.3065.
\]

To estimate the upper bound of \(|\Gamma_3|\) for
\(\tau_1\in\left(\frac{12}{35},1\right)\),
we next examine the quantity \(R(A,B,C)\) given in Lemma~\ref{L2}.

\item[(c)]
Next, observe that the inequality
\[
|C|(|B|+4|A|)\leq |AB|
\]
is equivalent to
\[
\tau_1\left(\frac{23}{6}\tau_1+\frac{35\tau_1^3}{3(1-\tau_1^2)}\right)
\leq
\frac{35\tau_1^3}{12(1-\tau_1^2)}\cdot\frac{23\tau_1}{6},
\]
which, after simplification, reduces to
\[
397\tau_1^2-552\geq 0.
\]
Since this inequality does not hold for any
\(\tau_1\in[0,1)\), the condition
\[
|C|(|B|+4|A|)\leq |AB|
\]
fails throughout the interval \([0,1)\).

\item[(d)]
Observe that the inequality
\[
|AB|\leq |C|(|B|-4|A|)
\]
is equivalent to
\[
\frac{805\tau_1^4}{72(\tau_1^2-1)}
\leq
\tau_1\left(\frac{23}{6}\tau_1-\frac{35\tau_1^3}{3(1-\tau_1^2)}\right),
\]
which, after simplification, reduces to
\[
1921\tau_1^2-276\leq 0.
\]
This inequality holds for
\[
\tau_1\leq \sqrt{\frac{276}{1921}}\approx 0.37904.
\]

Applying Lemma~\ref{L2} to \eqref{G12}, we obtain
\[
-|A|+|B|+|C|
=
\frac{29}{6}\tau_1
-\frac{35\tau_1^3}{12(1-\tau_1^2)}
=:G(\tau_1).
\]

Furthermore,
\[
G'(\tau_1)
=
\frac{29}{6}
-
\frac{35}{12}
\cdot
\frac{3\tau_1^2-\tau_1^4}{(1-\tau_1^2)^2}
>0,
\qquad
\forall\,\tau_1\in\left(\frac{12}{35},\sqrt{\frac{276}{1921}}\right).
\]
Hence, \(G\) is increasing on
\(\left(\frac{12}{35},\sqrt{\frac{276}{1921}}\right)\), and therefore
\[
\max_{\tau_1\in\left(\frac{12}{35},\,\sqrt{\frac{276}{1921}}\right)}
G(\tau_1)
=
G\!\left(\sqrt{\frac{276}{1921}}\right)
\approx
1.6464.
\]
Consequently,
\[
|\Gamma_3|
\leq
0.4116.
\]

\item[(e)]
For the interval
\[
\tau_1\in\left(\sqrt{\frac{276}{1921}},\,1\right),
\]
we apply the ``otherwise'' case of \(R(A,B,C)\) given in Lemma~\ref{L2}. In this range,
\[
R(A,B,C)
=
(|C|+|A|)
\sqrt{1-\frac{B^2}{4AC}}
=:H(\tau_1),
\]
where
\[
H(\tau_1)
=
\frac{(12+23\tau_1^2)\sqrt{529-109\tau_1^2}}
{12\sqrt{420}\,(1-\tau_1^2)}.
\]

Substituting this expression into \eqref{G12}, we obtain
\[
|\Gamma_3|
=
\frac{(1-\tau_1^2)}{4}H(\tau_1)
=
\frac{(12+23\tau_1^2)\sqrt{529-109\tau_1^2}}
{48\sqrt{420}}.
\]

To determine the maximum value of \(|\Gamma_3|\), let
\[
u=\tau_1^2,
\]
so that
\[
u\in\left(\frac{276}{1921},\,1\right).
\]
Define
\[
g(u)
=
(12+23u)\sqrt{529-109u}.
\]
Differentiating, we obtain
\begin{align*}
g'(u)
=\frac{23026-7521u}
{2\sqrt{529-109u}}.
\end{align*}

The critical point satisfies
\[
23026-7521u=0,
\]
that is,
\[
u=\frac{23026}{7521}\approx 3.06.
\]
Since this critical point lies outside the interval \((0,1)\), and
\[
23026-7521u>0,
\qquad u\in(0,1),
\]
it follows that
\[
g'(u)>0,
\qquad u\in(0,1).
\]
Hence \(g\) is strictly increasing on \((0,1)\). Therefore,
\(|\Gamma_3|\) attains its maximum value as \(\tau_1\to1^{-}\). Consequently,
\[
\max|\Gamma_3|
=
\lim_{\tau_1\to1^{-}}
\frac{(12+23\tau_1^2)\sqrt{529-109\tau_1^2}}
{48\sqrt{420}}
=
\frac{35}{48}
\approx
0.72917.
\]

Hence, from the above discussion, we conclude that
\[
|\Gamma_3|
\leq
\frac{35}{48}.
\]
The inequality is sharp, and equality is attained by the function \(f_1\), defined in \eqref{f1}.
\end{enumerate}
\end{enumerate}
\end{proof}

\subsection{{ \bf Bounds for the differences of logarithmic and logarithmic inverse coefficients of the class $\mathcal{P}^{*}$}}

In 1985, de Branges~\cite{LDB1} resolved the celebrated Bieberbach conjecture by showing that for a function
$f \in \mathcal{S}$ of the form~\eqref{eqn1}, the coefficient estimate $|a_n| \leq n$ holds for all $n \geq 2$,
with equality occurring only for the Koebe function
\[
k(z) := \frac{z}{(1-z)^2}
\]
and its rotations.
This fundamental result naturally led to the question of whether the inequality
\[
\bigl||a_{n+1}| - |a_n|\bigr| \leq 1, \quad n \geq 2,
\]
is valid for all functions in $\mathcal{S}$.
This problem was first investigated by Goluzin~\cite{G1} in connection with the Bieberbach conjecture.
Later, in 1963, Hayman~\cite{H} proved that
\[
\bigl||a_{n+1}| - |a_n|\bigr| \leq A,
\]
for every $f \in \mathcal{S}$, where $A \geq 1$ is an absolute constant.
The best estimate currently known is $A = 3.61$, due to Grinspan~\cite{G2}.
However, the sharp bound is known only in the case $n=2$
(see~\cite[Theorem~3.11]{PLD1}), namely
\[
-1 \leq |a_3| - |a_2| \leq 1.029\ldots
\]

For the starlike class $\mathcal{S}^*$, Pommerenke~\cite{Pommerenke1971} conjectured that
\[
\bigl||a_{n+1}| - |a_n|\bigr| \leq 1, \quad n \geq 2,
\]
which was later proved by Leung~\cite{Leung1978}.
For convex functions, Li and Sugawa~\cite{LiSugawa} studied the sharp upper bound of
$|a_{n+1}| - |a_n|$ for $n \geq 2$ and obtained sharp lower bounds for $n=2$ and $n=3$.
Several further results in this direction are available
(see~\cite{Peng2019, Arora2019, Arora2022, Arora2023}).

In 2023, Lecko and Partyka~\cite{Lecko2023} employed the Loewner method to establish sharp upper and lower
bounds for $|\gamma_2| - |\gamma_1|$ for functions in $\mathcal{S}$.
Their approach was subsequently simplified by Obradovi\'{c} and Tuneski~\cite{Obradovic2024}.
Moreover, Kumar and Cho~\cite{Kumar2023} obtained sharp estimates for
$|\gamma_2| - |\gamma_1|$ for certain subclasses of $\mathcal{S}$.

Motivated by these developments, this section is devoted to determining sharp lower and upper bounds for
$|\gamma_2| - |\gamma_1|$ and $|\Gamma_2| - |\Gamma_1|$ for functions belonging to the class
$\mathcal{P}^{*}$.

\begin{theo}\label{T3} Let \( f \in \mathcal{P}^{*} \), where \( \gamma_n \ (n = 1, 2) \) are defined by (\ref{log1}). Then 
\[
-\frac{1}{2} \leq |\gamma_2| - |\gamma_1| \leq \frac{1}{3}.
\]
The inequalities are sharp.
\end{theo}

\begin{proof}
Let $f\in \mathcal{P}^{*}$. In view of \eqref{log1}, \eqref{a2}, and \eqref{a3}, we deduce that
\begin{align}\label{lge1}
|\gamma_2| - |\gamma_1| &= \frac{1}{12}\big(|2c_2-c_1^2| - 3 |c_1|\big) \nonumber\\
&= \frac{1}{12} \Phi(c_1, c_2),
\end{align}
where 
\[
\Phi(c_1, c_2) = |K c_1^2 + L c_2| - |J c_1|,
\]
with \(K = -1\), \(L = 2\), and \(J = 3\). Moreover, we have $M = |4K + 2L| = | -4 + 4| = 0.$
We observe that \(|2K + L| = | -2 + 2| = 0\) and \(|L| + J = 2 + 3 = 5\), so that $|2K + L| < |L| + J.$ 
Furthermore, $J^2 = 9$ and $2|L|(M + 2|L|) = 4(0+4) = 16$, so the condition \(J^2 \leq 2|L|(M + 2|L|)\) is satisfied. By Lemma \ref{L6}, it follows that
\[
\Phi(c_1, c_2) \leq 2|L| = 4 \quad \text{and} \quad -\Phi(c_1, c_2) \leq 2J \sqrt{\frac{2|L|}{M + 2|L|}} = 6 \sqrt{\frac{4}{4}} = 6.
\]

Therefore, from \eqref{lge1}, we obtain 
\[
-\frac{1}{2} \leq |\gamma_2| - |\gamma_1| \leq \frac{1}{3}.
\]
The upper bound is sharp for the function \(f_2\) defined in (\ref{f2}), whereas the lower bound is attained for the function \(f_{1}\) defined in \eqref{f1}. 
\end{proof}

\begin{theo}\label{T4} Let \( f \in \mathcal{P}^{*} \) and \( \Gamma_n \ (n = 1, 2) \) be defined by (\ref{IG1}). Then 
\[
-\frac{1}{\sqrt{10}} \leq |\Gamma_2| - |\Gamma_1| \leq \frac{1}{3}.
\]
The inequalities are sharp.
\end{theo}

\begin{proof}
Let $f\in \mathcal{P}^{*}$. From \eqref{IG1}, \eqref{a2}, and \eqref{a3}, we get 
\begin{align}\label{IGe1}
|\Gamma_2| - |\Gamma_1| &=\frac{1}{24}\left| 5c_1^2-4c_2 \right|-\frac{1}{4}|c_1|\nonumber\\
 &=\frac{1}{24}\big(|5 c_1^2 - 4 c_2| - 6|c_1|\big) \nonumber\\
&= \frac{1}{24} \Phi(c_1, c_2),
\end{align}
where 
\[
\Phi(c_1, c_2) = |K c_1^2 + L c_2| - |J c_1|,
\]
with \(K = 5\), \(L = -4\), and \(J = 6\). In addition, we have $M = |4K + 2L| = |20 - 8| = 12.$
We note that \(|2K + L| = |10 - 4| = 6\) and \(|L| + J = 4 + 6 = 10\), implying that $|2K + L| \not\geq |L| + J.$ 
Moreover, $J^2 = 36 < 2|L|(M + 2|L|) = 8(12+8) = 160,$ so the condition \(J^2 < 2|L|(M + 2|L|)\) holds. By Lemma \ref{L6}, it follows that
\[
\Phi(c_1, c_2) \leq 2|L| = 8 \quad \text{and} \quad -\Phi(c_1, c_2) \leq 2J \sqrt{\frac{2|L|}{M + 2|L|}} = 12 \sqrt{\frac{8}{20}} = 12 \sqrt{\frac{2}{5}} = \frac{12}{\sqrt{10}}.
\]

Consequently, from \eqref{IGe1}, we obtain the sharp inequality
\[
-\frac{1}{\sqrt{10}} \leq |\Gamma_2| - |\Gamma_1| \leq \frac{1}{3}.
\]

Here, the upper bound is attained by the function \(f_2\) defined in \eqref{f2}, while the lower bound is realized by the function \(f_{5}\). For this extremal function, the Schwarz function is
\[
w(z)= \frac{z (2+\sqrt{7} z)}{\sqrt{7} + 2z},
\]
which satisfies $|w(z)|<1$ in $\mathbb{D}$ and $w(0)=0$. The corresponding Carath\'{e}odory function is
\[
p(z)=\frac{1+w(z)}{1-w(z)},
\]
which gives the desired equality. This implies $f_5$ is of the form: $f_5(z)= \int_{0}^{z} (1+\sin w(t))\, e^{w(t)}\, dt$.
\end{proof}

\subsection{{\bf Hankel determinants}}

In this section, we provide sharp bounds on the second-order Hankel determinant for the logarithmic coefficients for the class $\mathcal{P}^{*}$.

\begin{theo}\label{t1} Let $f\in \mathcal{P}^{*}$ and given by (\ref{eqn1}). Then 
\[\big| H_{2,1}\!\left(F_f/2\right)\big|\leq \frac{1}{9}.\]
The inequality is sharp.
\end{theo}

\begin{proof}
Since $f \in \mathcal{P}^{*}$. Then from (\ref{log1}), (\ref{Hank1}) and (\ref{a2}) we get
\begin{eqnarray}\label{ha1}
H_{2,1}\!\left(F_f/2\right) &=& \gamma_1\gamma_3 - \gamma_2^2\nonumber\\
&=& \frac{1}{4}\left(a_2a_4 - a_3^2 + \frac{1}{12}a_2^4\right)\nonumber\\
&=&-\frac{1}{4608}(5c_1^4 +4c_1^2c_2 -144c_1c_3 +128c_2^2).
\end{eqnarray}

Using Lemma \ref{L1} and (\ref{ha1}), a simple computation shows that

\begin{align}\label{ha2}
H_{2,1}\!\left(F_f/2\right) &= \frac{1}{8}\,\tau_1(1-\tau_1^2)(1-|\tau_2|^2)\tau_3
   - \frac{1}{72}(1-\tau_1^2)(8+\tau_1^2)\tau_2^{2} \notag\\
&\quad + \frac{1}{48}\tau_1^2(1-\tau_1^2)\tau_2
   - \frac{\tau_1^4}{96} \nonumber\\
&= \frac{1}{288}\Big(
   -3\tau_1^4
   + 6\tau_1^2\tau_2(1-\tau_1^2)
   - 4\tau_2^2(1-\tau_1^2)(8+\tau_1^2) \nonumber \\
&\qquad\qquad
   + 36(1-\tau_1^2)(1-|\tau_2|^2)\tau_1\tau_3
   \Big).
\end{align}

Since the class $\mathcal{P}$ and $H_{2,1}(F_{f}/2)$ are invariant under rotation, we may assume that $c_1\in [0,2]$ (see \cite[Theorem 3]{G}), that is in view of (\ref{p1}), $\tau_1\in [0,1]$.

Now we divide the following three possible cases:

{\bf Case 1.} Suppose that $\tau_1=1$. Then from (\ref{ha2}), we easily get
\[\big|H_{2,1}(F_{f}/2)\big|=\big|\gamma_1\gamma_3 - \gamma_2^2\big|=\frac{1}{96}.\]

{\bf Case 2.} Suppose that $\tau_1=0$. Then from (\ref{ha2}), we easily get
\[\big|H_{2,1}(F_{f}/2)\big|=\big|\gamma_1\gamma_3 - \gamma_2^2\big|=\frac{8}{72}|\tau_2^2|\leq \frac{1}{9}.\]

{\bf Case 3.} Suppose that $\tau_1\in (0,1)$. Then from (\ref{ha2}) we get
\begin{eqnarray}\label{ha3} \big|H_{2,1}(F_{f}/2)\big|&\leq &\frac{1}{288}\bigg( \big|3\tau_1^4-6(1-\tau_1^2)\tau_1^2\tau_2 +4(1-\tau_1^2)(8+\tau_1^2)\tau_2^2\big| \nonumber\\
&&\;\;\;\;\;\;+36(1-\tau_1^2)(1-|\tau_2|^2)\tau_1\tau_3\bigg)\nonumber\\
&\leq & \frac{1}{8}\tau_1(1-\tau_1^2)(|A+B\tau_2+C\tau_2^2|+1-|\tau_2|^2),\end{eqnarray}
where $A=\frac{\tau_1^3}{12(1-\tau_1^2)}$, $B=-\frac{1}{6}\tau_1$ and $C=\frac{8+\tau_1^2}{9\tau_1}$. So clearly $AC>0$. It is easy to see that 
\begin{eqnarray*}|B|-2(1-|C|)&=&\frac{1}{6}\tau_1+\frac{2(8+\tau_1^2)}{9\tau_1}-2\\
&=& \frac{7\tau_1^2-36\tau_1+32}{18\tau_1}.\end{eqnarray*}
Now note that $7\tau_1^2-36\tau_1+32>0$ for all $\tau_1\in (0,1)$. Hence $|B|>2(1-|C|)$.
Applying Lemma \ref{L2} in (\ref{ha3}), we get 
\begin{eqnarray*} \big|H_{2,1}(F_{f}/2)\big|&\leq & \frac{1}{8}\tau_1(1-\tau_1^2)(|A|+|B|+|C|)\\
& \leq & \frac{1}{8}\tau_1(1-\tau_1^2)\left(\frac{\tau_1^3}{12(1-\tau_1^2)}+\frac{\tau_1}{6}+\frac{8+\tau_1^2}{9\tau_1}\right)\nonumber \\
& \leq & \left|\frac{-7\tau_1^4-22\tau_1^2+32}{288}\right|
.\end{eqnarray*}
Let $\tau_1^2=t\in (0,1)$. Then $g(t)=-7t^2-22t+32$ and so $g'(t)=-2(11+7t)<0$. Therefore $g$ is a decreasing function on $(0,1)$. Hence $\max_{t\in (0,1)}\{g(t)\}=32$.
From above equation we deduce that 
\[\big|H_{2,1}(F_{f}/2)\big|\leq \frac{32}{288}=\frac{1}{9}.\]

The bound $\big|H_{2,1}(F_{f}/2)\big| \leq \tfrac{1}{9}$ is sharp, for the function $f_2$ defined in (\ref{f2}).
\end{proof}

\begin{theo}\label{t2} Let $f\in \mathcal{P}^{*}$ and given by (\ref{eqn1}). Then 
\[\big| H_{2,1}\!\left(F_{f^{-1}}/2\right)\big|\leq \frac{11}{96}.\]
The inequality is sharp.
\end{theo}

\begin{proof} Since $f \in \mathcal{P}^{*}$. Then from (\ref{IG1}), (\ref{Hank1}) and (\ref{a2}), we get 
\begin{eqnarray}\label{e2.1} H_{2,1}\!\left(F_{f^{-1}}/2\right)&=&\Gamma_1\Gamma_3 - \Gamma_2^2\nonumber\\
&=& \frac{c_1c_3}{32} - \frac{c_2^2}{36} - \frac{25c_1^2c_2}{1152} + \frac{79c_1^4}{4608}. \end{eqnarray}

By Lemma \ref{L1} and (\ref{e2.1}) a simple computation shows that

\begin{eqnarray}\label{e2.2}H_{2,1}\!\left(F_{f^{-1}}/2\right)=\frac{11}{96}\rho^4 - \frac{7}{48}\rho^2(1-\rho^2)\tau_2 - \frac{(1-\rho^2)}{8} \left[ \frac{8}{9}(1-\rho^2)\tau_2^2 + \rho^2 \tau_2^2 - \rho(1-|\tau_2|^2)\tau_3 \right].
\end{eqnarray}
Here $|\tau_1| = \rho \in [0, 1]$.
Since the class $\mathcal{P}$ and $H_{2,1}(F_{f^{-1}}/2)$ are invariant under rotation, we may assume that $c_1\in [0,2]$ (see \cite[Theorem 3]{G}), that is in view of (\ref{p1}), $\rho= \tau_1\in [0,1]$.

Now we divide the following three possible cases:

{\bf Case 1.} Suppose that $\tau_1=1$. Then from (\ref{e2.2}), we easily get
\[\big|H_{2,1}(F_{f^{-1}}/2)\big|=\big|\Gamma_1\Gamma_3 - \Gamma_2^2\big|= \frac{11}{96} \approx 0.1146.\]

{\bf Case 2.} Suppose that $\tau_1=0$. Then from (\ref{e2.2}), we easily get
\[\big|H_{2,1}(F_{f^{-1}}/2)\big|=\big|\Gamma_1\Gamma_3 - \Gamma_2^2\big|=\frac{1}{9} \approx 0.1111.\]

{\bf Case 3.} Suppose that $\tau_1\in (0,1)$. Then from (\ref{e2.2}) we get
\begin{eqnarray}\label{e2.3} \big|H_{2,1}(F_{f^{-1}}/2)\big|&\leq \frac{\rho(1-\rho^2)}{8} \left( |A + B\tau_2 + C\tau_2^2| + 1 - |\tau_2|^2 \right),\end{eqnarray}
where $A = \frac{11\rho^3}{12(1-\rho^2)}$, $B=-\frac{7}{6}\rho$ and $C=-\frac{8+\rho^2}{9\rho}$. Observe that $AC<0$. Hence we can apply case (ii) of Lemma \ref{L2}. Next, we check all the conditions of case (ii).

\begin{enumerate}

\item[(a)] Observe that
\[
1-|C|
=1-\frac{8+\rho^2}{9\rho}
=\frac{9\rho-8-\rho^2}{9\rho}
=\frac{-(\rho-1)(\rho-8)}{9\rho}.
\]
Since $\rho\in(0,1)$, it follows that $1-|C|<0$ for all $\rho\in(0,1)$.
Further, we have
\[
|B|-2(1-|C|)=|B|-2(\text{negative quantity}),
\]
which represents the sum of two positive terms and therefore cannot be negative.
Consequently, the inequality
\[
|B|<2(1-|C|)
\]
cannot be satisfied for any $\rho\in(0,1)$.

\item[(b)] 
We begin by analyzing the second term in the minimum expression. Consider the quantity
\[
-4AC(C^{-2}-1).
\]
Since $AC<0$, it follows that $-4AC>0$. For $\rho\in(0,1)$, we have
\[
|C|=\frac{8+\rho^2}{9\rho}.
\]
To verify whether $|C|>1$, observe that
\[
8+\rho^2>9\rho \quad \Longleftrightarrow \quad \rho^2-9\rho+8>0.
\]
The roots of the quadratic equation $\rho^2-9\rho+8=0$ are $1$ and $8$. Hence, for all
$\rho\in(0,1)$, the inequality holds, implying that $|C|>1$. Consequently, $C^2>1$ and therefore
$C^{-2}<1$, which shows that $(C^{-2}-1)<0$.

Since
\[
B^2=\frac{49\rho^2}{36}\geq 0 \quad \text{for all } \rho\in(0,1),
\]
it cannot be less than a negative quantity. Thus, the inequality
\[
B^2<-4AC(C^{-2}-1)
\]
is never satisfied. As a result, the condition
\[
B^2<\min\{4(1+|C|)^2,\,-4AC(C^{-2}-1)\}
\]
cannot hold for any $\rho\in(0,1)$.

\item[(c)] Observe that the inequality
\[
|C|(|B|+4|A|)-|AB|\leq 0
\]
is equivalent to
\[
171\rho^4-508\rho^2-224\geq 0.
\]
Setting $u=\rho^2$, the above inequality reduces to a quadratic in $u$, whose positive root yields
\[
\rho=\sqrt{3.36}\approx 1.833>1.
\]
Since $\rho\in(0,1)$, this condition cannot be satisfied. Therefore, the inequality
\[
|C|(|B|+4|A|)-|AB|\leq 0
\]
never holds for any admissible value of $\rho$.

\item[(d)] Note that the inequality $|AB| \leq |C|(|B|-4|A|)$ translates to: 
\[
\frac{77\rho^4}{72(1-\rho^2)} \leq \left( \frac{8+\rho^2}{9\rho} \right) \left( \frac{7\rho}{6} - \frac{11\rho^3}{3(1-\rho^2)} \right).
\]
Simplifying we get 
\[347\rho^4 + 900\rho^2 - 224 \leq 0.\] 
Let $x = \rho^2$. We solve the quadratic equation $347x^2 + 900x - 224 = 0$: 
\[x = \frac{-900 \pm \sqrt{900^2 - 4(347)(-224)}}{2(347)} = \frac{-900 \pm \sqrt{810000 + 310912}}{694} = \frac{-900 \pm \sqrt{1120912}}{694}.\]
Thus
\[\rho^2 \leq \frac{-900 + \sqrt{1120912}}{694} \approx 0.2287.\]
Hence this inequality holds true for $\rho \in [0, \rho_0]$, where $\rho_0 \approx 0.478$.

\begin{align*}\label{e2.4}
\bigl|H_{2,1}(F_{f^{-1}}/2)\bigr|
&= \Psi(\rho) \\
&= \frac{\rho(1-\rho^2)}{8}
\left[
-\frac{11\rho^3}{12(1-\rho^2)}
+ \frac{7\rho}{6}
+ \frac{8+\rho^2}{9\rho}
\right] \\
&= -\frac{79}{288}\rho^4 + \frac{7}{144}\rho^2 + \frac{1}{9}.
\end{align*}

Let
\begin{equation}
    g(t) = -\frac{79}{288}t^2 + \frac{7}{144}t + \frac{1}{9}.
\end{equation}
Differentiating $g(t)$ with respect to $t$, we obtain:
\begin{equation}
    g'(t) = -\frac{79}{144}t + \frac{7}{144}.
\end{equation}
Setting $g'(t) = 0$ yields the critical point:
\begin{equation}
    t_0 = \frac{7/144}{79/144} = \frac{7}{79}.
\end{equation}
We observe that $t_0 \approx 0.0886$. Note that $t_0 < \rho_0^2$ (since $0.0886 < 0.228$). Furthermore, since the coefficient of $t^2$ is negative ($-\frac{79}{288} < 0$), the function $g(t)$ represents a downward-opening parabola, implying the global maximum on the interval occurs at $t_0=\frac{7}{79}$.
Therefore 
\begin{align*}
    \Psi_{\max} = g\left(\frac{7}{79}\right) &= -\frac{79}{288}\left(\frac{7}{79}\right)^2 + \frac{7}{144}\left(\frac{7}{79}\right) + \frac{1}{9} \\
    &= \frac{859}{7584}.
\end{align*}
Consequently, we conclude that
\begin{equation}
    \bigl|H_{2,1}(F_{f^{-1}}/2)\bigr| \leq \frac{859}{7584} \approx 0.113.
\end{equation}

\item[(e)] Now consider the region \(\rho \in (\rho_0,1)\), where \(\rho_0 \approx 0.478\).
In this range, the extremum corresponding to the disk parameter \(\tau_2\)
is attained in the interior of the unit disk \(\mathbb{D}\).
Accordingly, we employ the functional $Y(A,B,C)=(|C|+|A|)\sqrt{1-\frac{B^2}{4AC}}.$

Let $\mathcal{K}(\rho)=-\frac{B^2}{4AC}.$ A direct computation yields
\[
\mathcal{K}(\rho)
= \frac{\left(-\frac{7\rho}{6}\right)^2}
{-4\left(\frac{11\rho^3}{12(1-\rho^2)}\right)
\left(-\frac{8+\rho^2}{9\rho}\right)}
= \frac{147(1-\rho^2)}{44(8+\rho^2)}.
\]
Hence,
\[
\sqrt{1-\frac{B^2}{4AC}}
= \sqrt{1+\mathcal{K}(\rho)}
= \sqrt{\frac{499-103\rho^2}{44(8+\rho^2)}}.
\]
Consequently,
\[
\bigl|H_{2,1}(F_{f^{-1}}/2)\bigr|
= \Psi(\rho)
= \frac{\rho(1-\rho^2)}{8} Y(A,B,C).
\]
Substituting the explicit expressions for \(A\), \(|B|\), and \(|C|\), we obtain
\[
\Psi_1(\rho)
= \frac{1}{8}
\left(
\frac{(8+\rho^2)(1-\rho^2)}{9}
+ \frac{11\rho^4}{12}
\right)
\sqrt{\frac{499-103\rho^2}{44(8+\rho^2)}}.
\]

Let $t = \rho^2$. Given the interval $\rho \in (\rho_0,1)$, we have $t \in (\rho_0^2,1)$ where $\rho_0^2 \approx 0.228$. We rewrite $\Psi(\rho)$ in terms of $t$ as follows:
\begin{align*}
\Psi_2(t) &= \frac{1}{8} \left( \frac{(8+t)(1-t)}{9} + \frac{11t^2}{12} \right) \sqrt{\frac{499-103t}{44(8+t)}} \\
&= \frac{1}{8} \left( \frac{32 - 28t + 29t^2}{36} \right) \sqrt{\frac{499-103t}{44(8+t)}} \\
&= \frac{1}{288\sqrt{44}} (29t^2 - 28t + 32) \sqrt{\frac{499-103t}{8+t}}.
\end{align*}

Differentiating with respect to $t$, we get 
\[
\Psi'(t) =\frac{\sqrt{11} \sqrt{\frac{499 - 103 t}{t + 8}} \left(11948 t^{3} + 70299 t^{2} - 518316 t + 265888\right)}{12672 \left(103 t^{2} + 325 t - 3992\right)}.
\]

\begin{figure}[H]
  \centering
  \includegraphics[width=0.5\linewidth]{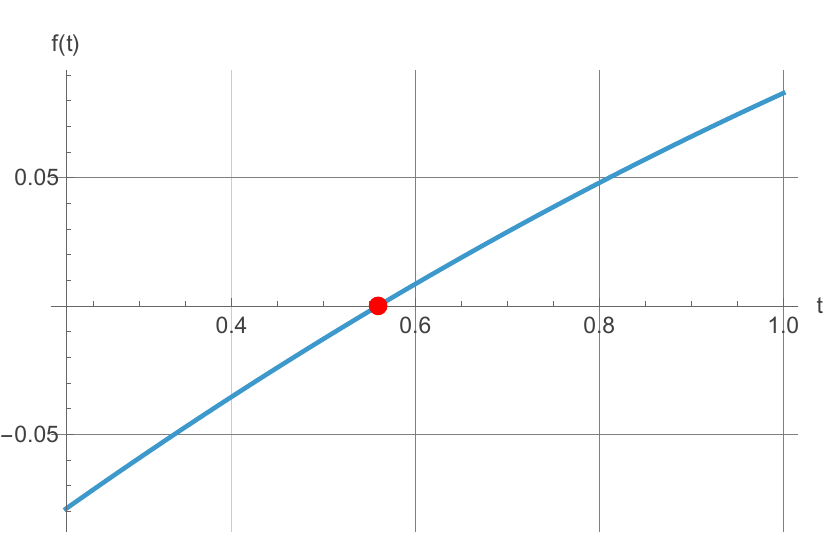}
  \caption{Graph of $\Psi'(t)$. The red mark indicates the point $p \approx 0.559475$ where $\Psi'(t)=0$.}
  \label{fig:psi_prime}
\end{figure}

From Figure~\ref{fig:psi_prime}, we can see that $\Psi(t)$ is decreasing on $(\rho_0^2, p)$ and increasing on $(p,1)$, where $p\approx 0.559475$ denotes the red mark in Figure 1. Therefore,
\[
\bigl|H_{2,1}(F_{f^{-1}}/2)\bigr| \leq \max\{\Psi_2(\rho_0^2),\Psi_2(1)\}=\Psi_2(1)=\frac{11}{96}.
\]

Therefore, the sharp upper bound for the second Hankel determinant of
logarithmic coefficients associated with the inverse function is
\[
\bigl|H_{2,1}(F_{f^{-1}}/2)\bigr|
\le \frac{11}{96}.
\]
The bound is sharp for the function $f_1$ defined in (\ref{f1}).
\end{enumerate}
\end{proof}

\subsection{{\bf Third-order Hermitian-Toeplitz determinant for the class $\mathcal{P}^{*}$}}

In this section, we establish the sharp upper and lower bounds of $T_{3,1}(f)$ for the class of functions $\mathcal{P}^{*}$.

\begin{theo}\label{T5} Let $f\in \mathcal{P}^{*}$ and given by (\ref{eqn1}). Then 
\[
-\frac{9}{35} \leq T_{3,1}(f) \leq 1.
\]
The inequalities are sharp.
\end{theo}

\begin{proof} Since $f \in \mathcal{P}^{*}$, proceeding in a similar manner as in Theorem \ref{T1}, we obtain $a_2$ and $a_3$ from (\ref{a2}) and (\ref{a3}), respectively.

It is important to note that both the class 
$\mathcal{P}$ of functions with positive real part and the class $\mathcal{P}^{*}$ 
are invariant under rotations. Hence, without loss of generality, since $|c_{n}| \leq 2$, 
we may assume $0 \leq c_{1} \leq 2$. 

In view of Lemma \ref{L1} together with (\ref{a2}), (\ref{a3}) and (\ref{HTexamples}), $T_{3,1}(f)$ can be written as
\begin{eqnarray}\label{tha3}  T_{3,1}(f) &=& 2\Re\!\left(a_2^2 \overline{a_3}\right) - 2|a_2|^2 - |a_3|^2 + 1\nonumber\\
&=& \frac{9(3c_1^4-32c_1^2+64)+24(4-c_1^2)c_1^2\Re\!(\tau_2)-16(4 - c_{1}^{2})^2|\tau_2|^2}{576}.\end{eqnarray}

Next, we aim to maximize the right-hand side of (\ref{tha3}). Since $\Re \tau_2\leq |\tau_2|$, it follows from (\ref{tha3}) that
\begin{eqnarray}\label{te.9} T_{3,1}(f) &\leq& \frac{1}{576}\left( 9(3c_1^4-32c_1^2+64)+24(4-c_1^2)c_1^2 |\tau_2|-16(4 - c_{1}^{2})^2\,|\tau_2|^2\right)\\
&=& \frac{1}{576} F(c_1^2, |\tau_2|)\nonumber.\end{eqnarray}
Setting $c_1^2=:x\in [0,4]$ and $|\tau_2|=:y\in [0,1]$, then $F(c_1^2, |\tau_2|)$ can be written as follows
\begin{eqnarray}\label{te.10} F(x,y)=9(3x^2-32x+64)+24(4-x)xy-16(4 - x)^2 y^2.\end{eqnarray}
Now, differentiating partially (\ref{te.10}) with respect to $x$ and $y$ we obtain 
\[
\frac{\partial F(x,y)}{\partial x} = 54x-288+96y-48xy+128y^2-32xy^2
\]
and 
\[
\frac{\partial F(x,y)}{\partial y} = 24x(4-x)-32(4-x)^2y.
\]
There are no critical points in the interior of the domain $(0,4)\times(0,1)$.
Therefore, the maximum and minimum value of $F(x,y)$ must occur on the boundaries of $[0,4] \times [0,1]$.

On the boundary of the rectangular region $[0,4] \times [0,1]$, the function $F(x,y)$ takes the following forms:
\[F(0,y)=576 - 256y^{2}\leq 576,\;\;\;F(4,y)=-144 \;\;\text{for all}\;y\in [0,1]\]
and
\[ F(x,0)= 27x^{2}-288x+576,\;\;\;F(x,1)= 320-64x-13x^2\leq 320 \;\;\text{for all}\;x\in [0,4].\]
Therefore the absolute maximum of $F(x,y)$ is $576$, occurring at the boundary point $(0,0)$.
Clearly $\max_{x\in [0,4]}F(x, 0)=576$.

From the above discussion, we deduce that 
\[T_{3,1}(f)\leq \frac{1}{576}\max\{ 576, -144, 320\}=1.\]

Next, we aim to minimize the right-hand side of (\ref{tha3}). Since $-\Re \tau_2\leq |\tau_2|$, it follows from (\ref{tha3}) that
\begin{eqnarray*}\label{te.11} T_{3,1}(f) &\geq& \frac{9(3c_1^4-32c_1^2+64)-24(4-c_1^2)c_1^2|\tau_2|-16(4 - c_{1}^{2})^2|\tau_2|^2}{576}\\
&\geq& \frac{1}{576} G(c_1^2, 1).\end{eqnarray*}
Setting $c_1^2=:x\in [0,4]$, then $G(x,1)$ can be written as follows
\begin{eqnarray*}\label{te.12} G(x,1)= 9(3x^2-32x+64)-24(4-x)x-16(4 - x)^2 = 35x^2-256x+320.\end{eqnarray*}
Here, we observe that $G'(x,1)=0$ for $x=\frac{128}{35}$. Moreover, since $G''(x,1)=70>0$, it follows that $G(x,1)$ attains its minimum at $x_{0}=\frac{128}{35}$, where  
\[
G(x_{0},1)=-\frac{5184}{35}.
\]
From the above discussion, we deduce that 
\[T_{3,1}(f)\geq \frac{-5184}{35 \times 576}=-\frac{9}{35}.\]
Hence proved.

The upper bound of $T_{3,1}(f)$ is sharp for the function $f_{3}$ defined in (\ref{f3}). For this function, $c_1=c_2=0$, $c_3=2$, giving $T_{3,1}(f)=1$.

Equality for the lower bound is attained for the function $f_6$, where 
\[f_6(z)= \int_{0}^{z} (1+\sin w(t))\, e^{w(t)}\,dt \]
with $w(z)=\frac{z\left(\sqrt{\frac{128}{35}}-2z\right)}{2-\sqrt{\frac{128}{35}}\,z}.$ This Schwarz function satisfies $w(0)=0$ and $|w(z)|<1$ in $\mathbb{D}$, and the corresponding Carath\'{e}odory function gives $c_1^2=128/35$ and $|\tau_2|=1$.
\end{proof}

\section*{Summary of Main Results and Extremal Functions}

\begin{table}[!ht]
\centering
\caption{Summary of sharp bounds and extremal functions for the class $\mathcal{P}^{*}$.}
\label{tab:summary}
\begin{tabular}{|c|c|c|c|}
\hline
\textbf{Result} & \textbf{Coefficient/Determinant} & \textbf{Sharp Bound} & \textbf{Extremal Function(s)} \\ \hline
Theorem~\ref{T1} (i) & $|\gamma_1|$ & $\displaystyle \frac{1}{2}$ & $f_1(z)=\displaystyle\int_0^z (1+\sin t)e^t\,dt$ \\ \hline
Theorem~\ref{T1} (ii) & $|\gamma_2|$ & $\displaystyle \frac{1}{3}$ & $f_2(z)=\displaystyle\int_0^z (1+\sin t^2)e^{t^2}\,dt$ \\ \hline
Theorem~\ref{T1} (iii) & $|\gamma_3|$ & $\displaystyle \frac{1}{4}$ & $f_3(z)=\displaystyle\int_0^z (1+\sin t^3)e^{t^3}\,dt$ \\ \hline
Theorem~\ref{T1} (iv) & $|\gamma_4|$ & $\displaystyle \frac{1}{5}$ & $f_4(z)=\displaystyle\int_0^z (1+\sin t^4)e^{t^4}\,dt$ \\ \hline
Theorem~\ref{T2} (i) & $|\Gamma_1|$ & $\displaystyle \frac{1}{2}$ & $f_1$ \\ \hline
Theorem~\ref{T2} (ii) & $|\Gamma_2|$ & $\displaystyle \frac{1}{2}$ & $f_1$ \\ \hline
Theorem~\ref{T2} (iii) & $|\Gamma_3|$ & $\displaystyle \frac{35}{48}$ & $f_1$ \\ \hline
Theorem~\ref{T3} & $|\gamma_2|-|\gamma_1|$ & $\displaystyle -\frac{1}{2} \le \cdot \le \frac{1}{3}$ & Lower: $f_1$; Upper: $f_2$ \\ \hline
Theorem~\ref{T4} & $|\Gamma_2|-|\Gamma_1|$ & $\displaystyle -\frac{1}{\sqrt{10}} \le \cdot \le \frac{1}{3}$ & Lower: $f_5$; Upper: $f_2$ \\ \hline
Theorem~\ref{t1} & $\left| H_{2,1}(F_f/2) \right|$ & $\displaystyle \frac{1}{9}$ & $f_2$ \\ \hline
Theorem~\ref{t2} & $\left| H_{2,1}(F_{f^{-1}}/2) \right|$ & $\displaystyle \frac{11}{96}$ & $f_1$ \\ \hline
Theorem~\ref{T5} & $T_{3,1}(f)$ & $\displaystyle -\frac{9}{35} \le \cdot \le 1$ & Lower: $f_6$; Upper: $f_3$ \\ \hline
\end{tabular}
\end{table}

\noindent Here the extremal functions are defined as follows:
\begin{align*}
f_1(z) &= \int_0^z (1+\sin t)e^t\,dt = z + z^2 + \frac12 z^3 + \frac18 z^4 + \cdots,\\
f_2(z) &= \int_0^z (1+\sin t^2)e^{t^2}\,dt = z + \frac23 z^3 + \frac{3}{10}z^5 + \cdots,\\
f_3(z) &= \int_0^z (1+\sin t^3)e^{t^3}\,dt = z + \frac12 z^4 + \frac{3}{14}z^7 + \cdots,\\
f_4(z) &= \int_0^z (1+\sin t^4)e^{t^4}\,dt = z + \frac25 z^5 + \frac16 z^9 + \cdots,\\
f_5(z) &= \int_0^z (1+\sin w(t))e^{w(t)}\,dt, \quad w(z)=\frac{z(2+\sqrt7 z)}{\sqrt7 + 2z},\\
f_6(z) &= \int_0^z (1+\sin w(t))e^{w(t)}\,dt, \quad w(z)=\frac{z\left(\sqrt{\frac{128}{35}}-2z\right)}{2-\sqrt{\frac{128}{35}}\,z}.
\end{align*}

For $f_5$ and $f_6$, the functions $w(z)$ are Schwarz functions (i.e., $w(0)=0$ and $|w(z)|<1$ in $\mathbb{D}$), which can be verified by the Schwarz-Pick lemma or direct computation.

\section*{Conclusion}

In this paper, we obtained sharp bounds for the logarithmic coefficients $\gamma_n$ ($n=1,2,3,4$) and the inverse logarithmic coefficients $\Gamma_n$ ($n=1,2,3$) for functions in the class $\mathcal{P}^{*}$. We also established sharp estimates for the differences $|\gamma_2|-|\gamma_1|$ and $|\Gamma_2|-|\Gamma_1|$, as well as for the second-order Hankel determinants associated with these coefficients. Finally, we determined sharp upper and lower bounds for the third-order Hermitian--Toeplitz determinant $T_{3,1}(f)$. All bounds are sharp, with extremal functions explicitly identified. These results extend and refine known estimates for related subclasses of analytic functions.

\section*{{\bf Declarations}}

\subsection*{Funding}
The first author acknowledges financial support from the Council of Scientific and Industrial Research (CSIR), New Delhi, India, under Grant No. 09/1224(16975)/2023-EMR-I.

\subsection*{Data Availability Statement}
Data sharing is not applicable to this article as no datasets were generated or analyzed during the current study.

\subsection*{Conflict of Interest}
The authors declare that they have no conflict of interest.

\subsection*{Author Contributions}
All authors contributed equally to this work and approved the final manuscript.

\end{document}